\documentclass{amsart}
\usepackage{latexsym,amssymb,amsmath,amsthm,amscd,graphicx,esint}

\makeatletter
\@namedef{subjclassname@2010}{%
  \textup{2010} Mathematics Subject Classification}
\makeatother

\numberwithin{equation}{section}
\newtheorem{theorem}{Theorem}[section]
\newtheorem{lemma}[theorem]{Lemma}

\newtheorem{proposition}[theorem]{Proposition}

\theoremstyle{definition}

\newtheorem{definition}[theorem]{Definition}

\theoremstyle{remark}

\newcommand{\cB}{{\mathcal B}}

\newcommand{\cD}{{\mathcal D}}

\newcommand{\cF}{{\mathcal F}}

\newcommand{\cQ}{{\mathcal Q}}

\newcommand{\R}{{\mathbb R}}
\newcommand{\T}{{\mathbb T}}
\newcommand{\Z}{{\mathbb Z}}

\def\al{\alpha}

\def\dl{\delta}
\def\eps{\varepsilon}

\def\lm{\lambda}

\def\sg{\sigma}

\def\0{\emptyset}
\def\6{\partial}
\def\8{\infty}

\def\l{\left}
\def\r{\right}
\def\ds{\displaystyle}

\begin{document}

\title[Two-weight norm inequalities on Morrey spaces]{Two-weight norm inequalities on Morrey spaces}

\author[H.~Tanaka]{Hitoshi Tanaka}
\address{Graduate School of Mathematical Sciences, The University of Tokyo, Tokyo, 153-8914, Japan}
\email{htanaka@ms.u-tokyo.ac.jp}

\thanks{
The author is supported by 
the FMSP program at Graduate School of Mathematical Sciences, the University of Tokyo, 
and Grant-in-Aid for Scientific Research (C) (No.~23540187), 
the Japan Society for the Promotion of Science. 
}



\date{}

\begin{abstract}
A description of all the admissible weights 
similar to the Muckenhoupt class $A_p$ 
is an open problem for the weighted Morrey spaces.
In this paper 
necessary condition and sufficient condition 
for two-weight norm inequalities on Morrey spaces to hold are given 
for the Hardy-Littlewood maximal operator. 
Necessary and sufficient condition is also verified for the power weights. 
\end{abstract}

\maketitle

\section{Introduction}\label{sec1}
The purpose of this paper is to develop a theory of weights 
for the Hardy-Littlewood maximal operator 
on the Morrey spaces. 
The Morrey spaces, 
which were introduced by C.~Morrey in order to study regularity questions which appear in the Calculus of Variations, 
describe local regularity more precisely than Lebesgue spaces 
and widely use not only harmonic analysis 
but also partial differential equations 
(cf. \cite{GT}). 

We shall consider all cubes in $\R^n$ which have their sides parallel to the coordinate axes.
We denote by $\cQ$ the family of all such cubes. 
For a cube $Q\in\cQ$ we use 
$l(Q)$ to denote the sides length of $Q$ 
and $|Q|$ to denote the volume of $Q$. 
Let $0<p<\8$ and $0<\lm<n$ 
be two real parameters. 
For $f\in L^p_{{\rm loc}}(\R^n)$, define 
$$
\|f\|_{L^{p,\lm}}
=
\sup_{Q\in\cQ}
\l(\frac1{l(Q)^{\lm}}\int_{Q}|f(x)|^p\,dx\r)^{1/p}.
$$
The Morrey space $L^{p,\lm}(\R^n)$ 
is defined to be the subset of all $L^p$ locally integrable functions $f$ on $\R^n$ 
for which $\|f\|_{L^{p,\lm}}$ is finite. 
It is easy see that 
$\|\cdot\|_{L^{p,\lm}}$
becomes the norm if $p\ge 1$ and 
becomes the quasi norm if $p\in(0,1)$. 
The completeness of Morrey spaces 
follows easily by that of Lebesgue spaces. 
Let $f$ be a locally integrable function on $\R^n$. 
The Hardy-Littlewood maximal operator $M$ is defined by 
$$
Mf(x)
=
\sup_{Q\in\cQ}
\fint_{Q}|f(y)|\,dy1_{Q}(x),
$$
where $\fint_{Q}f(x)\,dx$ stands for 
the usual integral average of $f$ over $Q$ 
and $1_{Q}$ denotes the characteristic function of the cube $Q$. 
By weights we will always mean 
non-negative, 
locally integrable functions 
which are positive on a set of positive measure. 
Given a measurable set $E$ and a weight $w$, 
$w(E)=\int_{E}w(x)\,dx$. 
Given $1<p<\8$, $p'=p/(p-1)$ 
will denote the conjugate exponent number of $p$. 
Let $0<p<\8$ and $w$ be a weight. 
We define the weighted Lebesgue space 
$L^p(\R^n,w)$ 
to be a Banach space equipped with the norm (or quasi norm) 
$$
\|f\|_{L^p(w)}
=
\l(\int_{\R^n}|f(x)|^pw(x)\,dx\r)^{1/p}
<\8.
$$
Let $0<p<\8$, $0<\lm<n$ and $w$ be a weight. 
We define the weighted Morrey space 
$L^{p,\lm}(\R^n,w)$ 
to be a Banach space equipped with the norm (or quasi norm) 
$$
\|f\|_{L^{p,\lm}(w)}
=
\sup_{Q\in\cQ}
\l(\frac1{l(Q)^{\lm}}\int_{Q}|f(x)|^pw(x)\,dx\r)^{1/p}
<\8.
$$

As is well-known, 
for the Hardy-Littlewood maximal operator $M$ 
and $p>1$, 
B.~Muckenhoupt \cite{Mu} showed that 
the weighted inequality 
$$
\|Mf\|_{L^p(w)}
\le C
\|f\|_{L^p(w)}
$$
holds if and only if 
$$
[w]_{A_p}
=
\sup_{Q\in\cQ}
\frac{w(Q)}{|Q|}
\l(\fint_{Q}w(x)^{-p'/p}\,dx\r)^{p/p'}
<\8.
$$
While, for $1<p\le q<\8$, 
E.~Sawyer \cite{Saw} showed that 
the weighted inequality 
$$
\|Mf\|_{L^q(u)}
\le C
\|f\|_{L^p(v)}
$$
holds if and only if 
$$
\l(\int_{Q}M[v^{-p'/p}1_{Q}](x)^qu(x)\,dx\r)^{1/q}
\le C
[v^{-p'/p}](Q)^{1/p}<\8,
$$
holds for every cube $Q\in\cQ$. 

For $p>1$ 
one says that a weight $w$ on $\R^n$ 
belongs to the Muckenhoupt class $A_p$ 
when $[w]_{A_p}<\8$. 
For $p=1$ 
one says that a weight $w$ on $\R^n$ 
belongs to the Muckenhoupt class $A_1$ 
when 
$$
[w]_{A_1}
=
\sup_{x\in\R^n}\frac{Mw(x)}{w(x)}
<\8.
$$
A description of all the admissible weights 
similar to the Muckenhoupt class $A_p$ 
is an open problem for the weighted Morrey space 
$L^{p,\lm}(\R^n,w)$ (see \cite{Sam2}). 
In \cite{ISST}, 
we proved the following partial answer to the problem. 

\begin{proposition}[{\rm\cite[Theorem 2.1]{ISST}}]\label{prp1.1}
Let $1<p<\8$, $0<\lm<n$ and 
$w$ be a weight. Then, 
for every cube $Q\in\cQ$, 
the weighted inequality 
$$
\l(\frac1{l(Q)^{\lm}}\int_{Q}Mf(x)^pw(x)\,dx\r)^{1/p}
\le C
\sup_{\substack{Q'\in\cQ \\ Q'\supset Q}}
\l(\frac1{l(Q')^{\lm}}\int_{Q'}|f(x)|^pw(x)\,dx\r)^{1/p}
$$
holds if and only if 
$$
\sup_{\substack{Q,Q'\in\cQ \\ Q\subset Q'}}
\frac{w(Q)}{l(Q)^{\lm}}
\frac{l(Q')^{\lm}}{|Q'|}
\l(\fint_{Q'}w(x)^{-p'/p}\,dx\r)^{p/p'}
<\8.
$$
\end{proposition}

This proposition says that 
the weighted inequality 
\begin{equation}\label{1.1}
\|Mf\|_{L^{p,\lm}(w)}
\le C
\|f\|_{L^{p,\lm}(w)}
\end{equation}
holds if 
\begin{equation}\label{1.2}
\sup_{Q\in\cQ}
\|w1_{Q}\|_{L^{1,\lm}}
\frac{l(Q)^{\lm}}{|Q|}
\l(\fint_{Q}w(x)^{-p'/p}\,dx\r)^{p/p'}
<\8.
\end{equation}
One sees that the power weights 
$w=|\cdot|^{\al}$ 
belong to 
the Muckenhoupt class $A_p$ 
if and only if 
$-n<\al<(p-1)n$. 
While, the power weights 
$w=|\cdot|^{\al}$ 
satisfy \eqref{1.2} 
if and only if 
$\lm-n\le\al<(p-1)n$. 
Let $H$ be the Hilbert transform defined by 
$$
Hf(x)
=
\lim_{\eps\to +0}
\frac{1}{\pi}\int_{\R}
\frac{1_{(\eps,\8)}(x-y)}{x-y}f(y)\,dy.
$$
For $1<p<\8$ and $0<\lm<1$, 
N.~Samko \cite{Sam1} showed that 
the weighted inequality 
$$
\|Hf\|_{L^{p,\lm}(w)}
\le C
\|f\|_{L^{p,\lm}(w)},
\quad w=|\cdot|^{\al},
$$
holds if and only if 
$\lm-1\le\al<\lm+(p-1)$. Thus, 
our sufficient condition \eqref{1.2} seems to be quite strong. 
In this paper 
we introduce another sufficient condition and necessary condition 
for which \eqref{1.1} to hold 
(Proposition \ref{prp4.1}). 
The conditions justify the power weights 
$w=|\cdot|^{\al}$ 
fulfill \eqref{1.1} 
if and only if 
$\lm-n\le\al<\lm+(p-1)n$ 
(Proposition \ref{prp4.2}). 
More precisely, in this paper 
we introduce sufficient condition and necessary condition 
for which two-weight Morrey norm inequalities to hold 
(Theorem \ref{thm3.1}),
which is closely related to Sawyer's two-weight theorem.
As an appendix, we show two-weight norm inequality 
in the upper triangle case 
$0<q<p<\8$, $1<p<\8$ 
(Proposition \ref{prp5.1}). 

The letter $C$ will be used for constants 
that may change from one occurrence to another. 
Constants with subscripts, such as $C_1$, $C_2$, do not change 
in different occurrences. 
By $A\approx B$ we mean that 
$c^{-1}B\le A\le cB$ with some positive constant $c$ independent of appropriate quantities.

\section{A dual equation}\label{sec2}
In this section we shall verify a dual equation of Morrey spaces (Lemma 2.4). 
For any measurable set $E\subset\R^n$ 
and any $f\in L^p(\R^n)$, we have 
$$
\int_{E}|f(x)|^p\,dx
\le\|f\|_{L^p}^p
<\8.
$$
While, if $f\in L^{p,\lm}(\R^n)$, 
then for any $Q\in\cQ$
$$
\int_{Q}|f(x)|^p\,dx
\le\|f\|_{L^{p,\lm}}^p
l(Q)^{\lm}.
$$
This implies that for any family of counterable cubes 
$\{Q_j\}\subset\cQ$ such that 
$E\subset\bigcup_jQ_j$, we have 
\begin{equation}\label{2.1}
\int_{E}|f(x)|^p\,dx
\le\sum_j\int_{Q_j}|f(x)|^p\,dx
\le\|f\|_{L^{p,\lm}}^p
\sum_jl(Q_j)^{\lm}.
\end{equation}
In general, 
if $E\subset\R^n$ and $0<\al\le n$, 
then the $\al$-dimensional Hausdorff content of $E$ is defined by
$$
H^{\al}(E)
=
\inf\l\{\sum_jl(Q_j)^{\al}\r\},
$$
where the infimum is taken over all coverings of $E$ 
by countable families of cubes $\{Q_j\}\subset\cQ$. 
Thanks to this definition, we get by \eqref{2.1} 
\begin{equation}\label{2.2}
\int_{E}|f(x)|^p\,dx
\le\|f\|_{L^{p,\lm}}^p
H^{\lm}(E).
\end{equation}
The Choquet integral of $\phi\ge 0$ 
with respect to the Hausdorff content $H^{\al}$ is defined by 
$$
\int_{\R^n}\phi\,dH^{\al}
=
\int_0^{\8}
H^{\al}(\{y\in\R^n:\,\phi(y)>t\})
\,dt.
$$
Thus, by \eqref{2.2}, 
for any $\phi\ge 0$ and 
any $f\in L^{p,\lm}(\R^n)$, 
\begin{equation}\label{2.3}
\int_{\R^n}|f(x)|^p\phi(x)\,dx
=
\int_0^{\8}
\int_{\{y\in\R^n:\,\phi(y)>t\}}|f(x)|^p\,dx
\,dt
\le\|f\|_{L^{p,\lm}}^p
\int_{\R^n}\phi\,dH^{\lm}.
\end{equation}

\begin{definition}\label{def2.1}
Let $0<\lm<n$. Define 
the basis $\cB_{\lm}$ 
to be the set of all weights $b$ such that 
$b\in A_1$ and 
$\int_{\R^n}b\,dH^{\lm}\le 1$.
\end{definition}

Following the argument in \cite{AX2}, 
we introduce another characterization of the Morrey space by \eqref{2.3}. 
We need the following lemma. 

\begin{lemma}[{\rm\cite[Lemma 1]{OV}}]\label{lem2.2}
Let $0<\al<n$ and $p>\al/n$. 
Then, 
for some constant $C$ depending only on 
$\al$, $n$ and $p$, 
$$
\int_{\R^n}M[1_{Q}]^p\,dH^{\al}
\le C
l(Q)^{\al}.
$$
\end{lemma}

Let $0<\lm<\lm_0<n$ and 
$f\in L^{p,\lm}(\R^n)$. 
It follows from \eqref{2.3} and 
Lemma \ref{lem2.2} that, 
for every cube $Q\in\cQ$, 
\begin{align*}
\frac1{l(Q)^{\lm}}\int_{Q}|f(x)|^p\,dx
&=
\frac1{l(Q)^{\lm}}\int_{\R^n}|f(x)|^p1_{Q}(x)\,dx
\\ &\le
\frac1{l(Q)^{\lm}}\int_{\R^n}|f(x)|^pM[1_{Q}](x)^{\lm_0/n}\,dx
\\ &\le 
\|f\|_{L^{p,\lm}}^p
\frac1{l(Q)^{\lm}}
\int_{\R^n}M[1_{Q}]^{\lm_0/n}\,dH^{\lm}
\\ &\le C
\|f\|_{L^{p,\lm}}^p,
\end{align*}
which yields 
\begin{equation}\label{2.4}
\|f\|_{L^{p,\lm}}
\approx
\sup_{b\in\cB_{\lm}}
\l(\int_{\R^n}|f(x)|^pb(x)\,dx\r)^{1/p},
\end{equation}
where we have used the fact that 
$M[1_{Q}]^{\lm_0/n}\in A_1$, 
since $\lm_0/n<1$ 
(cf. \cite[Chapter II]{GR}).

\begin{definition}[{\rm\cite{AX2}}]\label{def2.3}
Let $1<p<\8$ and $0<\lm<n$. 
The space $H^{p,\lm}(\R^n)$ 
is defined by the set of all measurable functions $f$ on $\R^n$ 
with the quasi norm 
$$
\|f\|_{H^{p,\lm}}
=
\inf_{b\in\cB_{\lm}}
\l(\int_{\R^n}|f(x)|^pb(x)^{-p/p'}\,dx\r)^{1/p}
<\8.
$$
\end{definition}

For non-negative functions 
$f\in L^{p,\lm}(\R^n)$ 
and 
$g\in H^{p',\lm}(\R^n)$, 
there holds by H\"{o}lder's inequality that 
\begin{align}\label{2.5}
\int_{\R^n}f(x)g(x)\,dx
&=
\int_{\R^n}f(x)b(x)^{1/p}g(x)b(x)^{-1/p}\,dx
\\ \nonumber &\le
\l(\int_{\R^n}f(x)^pb(x)\,dx\r)^{1/p}
\l(\int_{\R^n}g(x)^{p'}b(x)^{-p'/p}\,dx\r)^{1/p'}
\\ \nonumber &\le C
\|f\|_{L^{p,\lm}}\|g\|_{H^{p',\lm}},
\quad b\in\cB_{\lm}.
\end{align}
In this section we shall verify the following lemma.

\begin{lemma}\label{lem2.4}
Let $1<p<\8$ and $0<\lm<n$. 
Then, 
for any measurable function $g$ on $\R^n$,
we have the estimate (allowing to be infinite) 
$$
\|g\|_{H^{p',\lm}}
\approx
\sup_{f}\int_{R^n}|f(x)g(x)|\,dx,
$$
where the supremum is taken over all functions 
$f\in L^{p,\lm}(\R^n)$ with 
$\|f\|_{L^{p,\lm}}\le 1$.
\end{lemma}

This lemma was first introduced in \cite{AX2} without the proof. 
In \cite{ISY}, 
T.~Izumi et al give the full proof for the block spaces on the unit circle $\T$ with the help of Functional Analysis. 
In \cite{ST}, 
we give the proof for the block spaces on the Euclidean space $\R^n$. 

\begin{definition}\label{def2.5}
Let $1<p<\8$ and $0<\lm<n$. 
The block space $B^{p,\lm}(\R^n)$ 
is defined by the set of all measurable functions $f$ on $\R^n$ 
with the norm
$$
\|f\|_{B^{p,\lm}}
=
\inf\l\{
\|\{c_k\}\|_{l^1}:\,
f=\sum_kc_ka_k
\r\}<\8,
$$
where $a_k$ is a $(p,\lm)$-atom and 
$\|\{c_k\}\|_{l^1}=\sum_k|c_k|<\8$, 
and the infimum is taken over all possible atomic decompositions of $f$. 
Additionally, 
we say that a function $a$ on $\R^n$ 
is a $(p,\lm)$-atom provided that 
$a$ is supported on a cube $Q\in\cQ$ 
and satisfies 
$$
\|a\|_{L^p}
\le
\frac1{l(Q)^{\lm/p'}}.
$$
\end{definition}

\begin{lemma}[{\rm\cite{ST}}]\label{lem2.6}
Let $1<p<\8$ and $0<\lm<n$. 
Then, 
for any measurable function $g$ on $\R^n$,
we have the estimate (allowing to be infinite) 
$$
\|g\|_{B^{p',\lm}}
=
\sup_{f}\int_{R^n}|f(x)g(x)|\,dx,
$$
where the supremum is taken over all functions 
$f\in L^{p,\lm}(\R^n)$ with 
$\|f\|_{L^{p,\lm}}\le 1$.
\end{lemma}

\begin{proof}[Proof of Proposition \ref{lem2.4}] 
Thanks to Lemma \ref{lem2.6}, 
we need only verify that 
$H^{p,\lm}(\R^n)
=
B^{p,\lm}(\R^n)$
with 
$\|\cdot\|_{H^{p,\lm}}
\approx
\|\cdot\|_{B^{p,\lm}}$. 
This fact was proved in \cite{AX1}. 
But, the direct proof is given here for the completeness. 

We will denote by $\cD$ the family of all dyadic cubes 
$Q=2^{-k}(m+[0,1)^n)$, 
$k\in\Z,\,m\in\Z^n$. 
Assume that for non-negative function 
$f\in H^{p,\lm}(\R^n)$, 
\begin{equation}\label{2.6}
\l(\int_{\R^n}f(x)^pb(x)^{-p/p'}\,dx\r)^{1/p}
\le
2\|f\|_{H^{p,\lm}},
\quad b\in\cB_{\lm}.
\end{equation}
Consider 
$E_k=\{x\in\R^n:\,b(x)>2^k\}$, 
$k\in\Z$. Then, 
\begin{equation}\label{2.7}
\int_{\R^n}b\,dH^{\lm}
\approx
\sum_k2^kH^{\lm}(E_k)
\approx 1.
\end{equation}
By the definition of 
the Hausdorff content $H^{\lm}$ 
and its dyadic equivalence (cf. \cite{OV}), 
one can select a set of the pairwise disjoint dyadic cubes 
$\{Q_{k,j}\}\subset\cD$
such that 
$E_k\subset\bigcup_jQ_{k,j}$ and 
\begin{equation}\label{2.8}
\sum_jl(Q_{k,j})^{\lm}
\le
2H^{\lm}(E_k).
\end{equation}
Upon defining 
$$
\dl_{k,j}
=
Q_{k,j}\setminus\bigcup_iQ_{k+1,i},
$$
we see that the sets $\dl_{k,j}$ are pairwise disjoint and 
$\R^n=\bigcup_{k,j}\dl_{k,j}$. 
With this, we obtain 
$$
f=\sum_{k,j}c_{k,j}a_{k,j},
$$
where
$$
c_{k,j}
=
l(Q_{k,j})^{\lm/p'}
\l(\int_{\dl_{k,j}}f(x)^p\,dx\r)^{1/p}
$$
and
$$
a_{k,j}
=
l(Q_j,k)^{-\lm/p'}
\l(\int_{\dl_{k,j}}f(x)^p\,dx\r)^{-1/p}
f1_{\dl_{k,j}}.
$$
It is easy to check that each $a_{k,j}$ is a $(p,\lm)$-atom. 
To prove that $f\in B^{p,\lm}(\R^n)$, 
it remains to verify that 
$\{c_{k,j}\}$ is summable. 

Notice that 
$b(x)\le 2^{k+1}$ if $x\in\dl_{k,j}$. 
This yields, 
by using H\"{o}lder's inequality, 
\begin{align*}
\|\{c_{k,j}\}\|_{l^1}
&\le C
\sum_{k,j}
l(Q_{k,j})^{\lm/p'}
2^{k/p'}
\l(\int_{\dl_{k,j}}f(x)^pb(x)^{-p/p'}\,dx\r)^{1/p}
\\ &\le C
\l(\sum_{k,j}l(Q_{k,j})^{\lm}2^k\r)^{1/p'}
\l(\int_{\R^n}f(x)^pb(x)^{-p/p'}\,dx\r)^{1/p}
\\ &\le C
\l(\sum_k2^kH^{\lm}(E_k)\r)^{1/p'}
\l(\int_{\R^n}f(x)^pb(x)^{-p/p'}\,dx\r)^{1/p}
\\ &\le C
\|f\|_{H^{p,\lm}},
\end{align*}
where we have used 
\eqref{2.6}--\eqref{2.8}.
This proves 
$H^{p,\lm}(\R^n)
\subset
B^{p,\lm}(\R^n)$
with 
$\|\cdot\|_{B^{p,\lm}}
\le C
\|\cdot\|_{H^{p,\lm}}$.

We now prove converse. 
Suppose that $f\in B^{p,\lm}(\R^n)$. 
So, $f=\sum_jc_ja_j$ with 
$\{c_j\}\in l^1$ and 
each $a_j$ is a $(p,\lm)$-atom. 
Assume that $Q_j$ is the support cube of $a_j$. 
For $0<\lm<\lm_0<n$, define 
$$
b(x)
=
\|\{c_j\}\|_{l^1}^{-1}
\sum_j|c_j|b_j(x)
$$
with
$$
b_j(x)
=
\frac1{l(Q_j)^{\lm}}
M[1_{Q_j}](x)^{\lm_0/n}.
$$
Then, we see that 
$$
\int_{\R^n}b_j\,dH^{\lm}\le C
\text{ and }
[b_j]_{A_1}\le C.
$$
This means that 
$$
\int_{\R^n}b\,dH^{\lm}\le C
\text{ and }
[b]_{A_1}\le C.
$$
Thus, we have $Cb\in\cB_{\lm}$ and 
$$
\l(\int_{\R^n}|f(x)|^pb(x)^{-p/p'}\,dx\r)^{1/p}
\le
\sum_j|c_j|
\l(\int_{Q_j}|a_j(x)|^pb(x)^{-p/p'}\,dx\r)^{1/p}.
$$
Notice that whenever $x\in Q_j$ 
$$
b(x)^{-p/p'}
\le b_j(x)^{-p/p'}
\le l(Q_j)^{\lm{p/p'}},
$$
which implies
$$
\l(\int_{\R^n}|f(x)|^pb(x)^{-p/p'}\,dx\r)^{1/p}
\le
\sum_j|c_j|
l(Q_j)^{\lm/p'}
\l(\int_{Q_j}|a_j(x)|^p\,dx\r)^{1/p}
\le
\sum_j|c_j|.
$$
This proves 
$B^{p,\lm}(\R^n)
\subset
H^{p,\lm}(\R^n)$
with 
$\|\cdot\|_{H^{p,\lm}}
\le C
\|\cdot\|_{B^{p,\lm}}$. 
These complete the proof of Lemma \ref{2.4}.
\end{proof}

\section{Two-weight norm inequalities}\label{sec3}
In this section we shall prove the following theorem. 

\begin{theorem}\label{thm3.1}
Let $1<p<\8$, $0<q<\8$, 
$0<\lm<n$, $u,v$ be weights. 
Consider the following five statements: 

\begin{enumerate}
\item[{\rm(a)}] 
There exists a constant $c_1>0$ 
such that 
$$
\|Mf\|_{L^{q,\lm}(u)}
\le c_1
\|f\|_{L^{p,\lm}(v)}
$$
holds for every function 
$f\in L^{p,\lm}(\R^n,v)$;
\item[{\rm(b)}] 
There exists a constant $c_2>0$ 
such that 
$$
\frac1{|Q|}
\|u^{1/q}1_{Q}\|_{L^{q,\lm}}
\|v^{-1/p}1_{Q}\|_{H^{p',\lm}}
\le c_2
$$
holds for every cube $Q\in\cQ$;
\item[{\rm(c)}] 
There exists a constant $c_3>0$ 
such that 
$$
\inf_{b\in\cB_{\lm}}
\l(
\frac1{l(Q_0)^{\lm/q}}
\sup_{\substack{Q\in\cQ \\ Q\subset Q_0}}
\frac1{[(bv)^{-p'/p}](Q)^{1/p}}
\l(\int_{Q}M[(bv)^{-p'/p}1_{Q}](x)^qu(x)\,dx\r)^{1/q}
\r)
\le c_3
$$
holds for every cube $Q_0\in\cQ$;
\item[{\rm(d)}] 
There exists a constant $c_4>0$ 
such that, for some $a>1$, 
$$
\inf_{b\in\cB_{\lm}}
\l(
\frac1{l(Q_0)^{\lm/q}}
\sup_{\substack{Q\in\cQ \\ Q\subset Q_0}}
\frac{u(Q)^{1/q}}{|Q|^{1/p}}
\l(\fint_{Q}[b(x)v(x)]^{-ap'/p}\,dx\r)^{1/ap'}
\r)
\le c_4
$$
holds for every cube $Q_0\in\cQ$;
\item[{\rm(e)}] 
There exists a constant $c_5>0$ 
such that 
$$
\inf_{b\in\cB_{\lm}}
\l(
\frac1{l(Q_0)^{\lm/q}}
\l(\int_{Q_0}
M[(bv)^{-p'/p}](x)^r
u(x)^{r/q}[b(x)v(x)]^{rp'/p^2}
\,dx\r)^{1/r}
\r)
\le c_5,
\quad\frac1q=\frac1r+\frac1p,
$$
holds for every cube $Q_0\in\cQ$.
\end{enumerate}

\noindent
Then,

\begin{enumerate}
\item[{\rm(I)}] 
{\rm(a)} implies {\rm(b)} 
with $c_2\le Cc_1$;
\item[{\rm(II)}] 
When $1<p\le q<8$, 
{\rm(b)} and {\rm(c)} imply {\rm(a)} 
with $c_1\le C(c_2+c_3)$;
\item[{\rm(III)}] 
When $1<p\le q<8$, 
{\rm(b)} and {\rm(d)} imply {\rm(a)} 
with $c_1\le C(c_2+c_4)$;
\item[{\rm(IV)}] 
When $0<q<p<8$ and $1<p<\8$, 
{\rm(b)} and {\rm(e)} imply {\rm(a)} 
with $c_1\le C(c_2+c_5)$.
\end{enumerate}
\end{theorem}

We shall prove this theorem in the remainder of this section. 
Recall that $\cD$ denotes the family of all dyadic cubes 
$Q=2^{-k}(m+[0,1)^n)$, 
$k\in\Z,\,m\in\Z^n$. 
In the following proof, 
by the argument which uses appropriate averages of the sifted dyadic cubes,
we can replace the set of cubes $\cQ$ by the set of dyadic cubes $\cD$ 
(cf. \cite{Hy1}). So, 
the Hardy-Littlewood maximal operator $M$ 
can be replaced by 
the dyadic Hardy-Littlewood maximal operator $M_d$. 
But, for the sake of simplicity, 
we will denote $M_d$ by the same $M$. 

\subsection{Proof of Theorem \ref{thm3.1} (I)}\label{ssec3.1}
Assume that the statement (a). Then, 
$$
\|Mf\|_{L^{q,\lm}(u)}
\le c_1
\|f\|_{L^{p,\lm}(v)}
$$
holds for every function 
$f\in L^{p,\lm}(\R^n,v)$.
For any cube $Q\in\cD$ and 
any function $f\in L^{p,\lm}(\R^n,v)$, 
$$
\fint_{Q}|f(x)|\,dx
\|u^{1/q}1_{Q}\|_{L^{q,\lm}}
=
\l\|\fint_{Q}|f(x)|\,dxu^{1/q}1_{Q}\r\|_{L^{q,\lm}}
\le
\|M[f1_{Q}]\|_{L^{q,\lm}(u)}
\le c_1
\|f1_{Q}\|_{L^{p,\lm}(v)}.
$$
Taking the supremum over all functions $f$ 
with 
$\|f1_{Q}\|_{L^{p,\lm}(v)}\le 1$, 
we have by Lemma \ref{lem2.4} 
$$
\frac1{|Q|}
\|u^{1/q}1_{Q}\|_{L^{q,\lm}}
\|v^{-1/p}1_{Q}\|_{H^{p',\lm}}
\le Cc_1,
$$
which is the statement (b).

\subsection{Proof of Theorem \ref{thm3.1} (II)}\label{ssec3.2}
We need more a lemma. 
Let $\mu$ be a positive measure on $\R^n$ and 
$f$ be a locally $\mu$-integrable function on $\R^n$. 
The dyadic Hardy-Littlewood maximal operator $M_{\mu}$ is defined by 
$$
M_{\mu}f(x)
=
\sup_{Q\in\cD}
\fint_{Q}|f(y)|\,d\mu(y)1_{Q}(x).
$$

\begin{lemma}[{\rm\cite{Hy1}}]\label{lem3.2}
We have the estimate 
$$
\|M_{\mu}f\|_{L^p(\mu)}
\le p'
\|f\|_{L^p(\mu)},
\quad p\in(1,\8].
$$
\end{lemma}

Assume that $1<p\le q<\8$ and 
the statements (b) and (c). 
Without loss of generality 
we may assume that $f$ is non-negative. 
Recall that $M$ is now the dyadic Hardy-Littlewood maximal operator. 
Fix a cube $Q_0$ in $\cD$. 
Then, by a standard argument we have 
$$
Mf(x)
\le 
C_{\8}+M[f1_{Q_0}](x),
\quad x\in Q_0,
$$
with
$$
C_{\8}
=
\sup_{\substack{Q\in\cD \\ Q\supset\neq Q_0}}
\fint_{Q}f(y)\,dy.
$$
By the definition of the weighted Morrey norm, 
we have to evaluate two quantities: 
\begin{equation}\label{3.1}
C_{\8}
\l(\frac1{l(Q_0)^{\lm}}\int_{Q_0}u(x)\,dx\r)^{1/q};
\end{equation}
\begin{equation}\label{3.2}
\l(\frac1{l(Q_0)^{\lm}}\int_{Q_0}M[f1_{Q_0}](x)^qu(x)\,dx\r)^{1/q}.
\end{equation}

\noindent{\bf The estimate of \eqref{3.1}}\ \ 
There holds 
$$
\l(\frac1{l(Q_0)^{\lm}}\int_{Q_0}u(x)\,dx\r)^{1/q}
\le
\|u^{1/q}1_{Q_0}\|_{L^{q,\lm}},
$$
and, for $Q\in\cD$ 
such that $Q\supset\neq Q_0$, 
$$
\fint_{Q}f(y)\,dy
\le
\frac1{|Q|}
\|v^{-1/p}1_{Q}\|_{H^{p',\lm}}
\|f1_{Q}\|_{L^{p,\lm}(v)},
$$
where we have used \eqref{2.5}.
These yield by use of the statement (b) 
$$
{\rm\eqref{3.1}}
\le c_2
\|f\|_{L^{p,\lm}(v)}.
$$

\noindent{\bf The estimate of \eqref{3.2}}\ \ 
Let 
$\cD(Q_0)=\{Q\in\cD:\,Q\subset Q_0\}$. 
Consider, for all $Q\in\cD(Q_0)$, 
$$
E(Q)
=
\l\{x\in Q:\,
M[f1_{Q_0}](x)=\fint_{Q}f(y)\,dy
\r\}
\setminus
\bigcup_{\substack{Q'\in\cD(Q_0) \\ Q'\supset\neq Q}}
\l\{x\in Q':\,
M[f1_{Q_0}](x)=\fint_{Q'}f(y)\,dy
\r\}.
$$
A little thought confirms that 
the sets $E(Q)$ are pairwise disjoint and 
$$
M[f1_{Q_0}](x)
=
\sum_{Q\in\cD(Q_0)}
\fint_{Q}f(y)\,dy1_{E(Q)}(x),
\quad x\in Q_0.
$$
Take a function $g$ which is non-negative, 
is supported on $Q_0$ and satisfies 
$\|g\|_{L^{q'}(u)}\le 1$.
Upon using the duality argument, 
we shall estimate 
\begin{equation}\label{3.3}
\sum_{Q\in\cD(Q_0)}
\fint_{Q}f(y)\,dy
\int_{E(Q)}g(x)u(x)\,dx.
\end{equation}
Fix $b\in\cB_{\lm}$ so that 
\begin{equation}\label{3.4}
\frac1{l(Q_0)^{\lm/q}}
\sup_{\substack{Q\in\cQ \\ Q\subset Q_0}}
\frac1{[(bv)^{-p'/p}](Q)^{1/p}}
\l(\int_{Q}M[(bv)^{-p'/p}1_{Q}](x)^qu(x)\,dx\r)^{1/q}
\le 2c_3,
\end{equation}
and, let $\sg=(bv)^{-p'/p}$. 
Then, \eqref{3.3} can be rewritten as 
$$
\sum_{Q\in\cD(Q_0)}
\frac{\sg(Q)}{|Q|}
\fint_{Q}f(y)\sg(y)^{-1}\,d\sg(y)
\int_{E(Q)}g(x)u(x)\,dx,
$$
where $d\sg(y)$ denotes $\sg(y)\,dy$.
We now employ the argument of the principal cubes 
(cf. \cite{Hy2,Ta}).

We define the collection of principal cubes 
$$
\cF=\bigcup_{k=0}^{\8}\cF_k,
$$
where $\cF_0=\{Q_0\}$, 
$$
\cF_{k+1}
=
\bigcup_{F\in\cF_k}ch_{\cF}(F)
$$
and $ch_{\cF}(F)$ is defined by 
the set of all maximal dyadic cubes $Q\subset F$ 
such that 
$$
\fint_{Q}f(y)\sg(y)^{-1}\,d\sg(y)
>
2\int_{F}f(y)\sg(y)^{-1}\,d\sg(y).
$$
Observe that 
$$
\sum_{F'\in ch_{\cF}(F)}\sg(F')
\le
\l(2\fint_{F}f(y)\sg(y)^{-1}\,d\sg(y)\r)^{-1}
\sum_{F'\in ch_{\cF}(F)}
\int_{F'}f(y)\sg(y)^{-1}\,d\sg(y)
\le
\frac{\sg(F)}{2},
$$
and, hence, 
\begin{equation}\label{3.5}
\sg(E_{\cF}(F))
=
\sg\l(F\setminus\bigcup_{F'\in ch_{\cF}(F)}F'\r)
\ge
\frac{\sg(F)}{2},
\end{equation}
where the sets $E_{\cF}(F)$ are pairwise disjoint. 
We further define the stopping parents 
$$
\pi_{\cF}(Q)
=
\min\{F\supset Q:\,F\in\cF\}.
$$

It follows that 
\begin{align*}
{\rm\eqref{3.3}}
&=
\sum_{F\in\cF}
\sum_{\substack{Q: \\ \pi(Q)=F}}
\frac{\sg(Q)}{|Q|}
\fint_{Q}f(y)\sg(y)^{-1}\,d\sg(y)
\int_{E(Q)}g(x)u(x)\,dx
\\ &\le 2
\sum_{F\in\cF}
\fint_{F}f(y)\sg(y)^{-1}\,d\sg(y)
\sum_{\substack{Q: \\ \pi(Q)=F}}
\frac{\sg(Q)}{|Q|}
\int_{E(Q)}g(x)u(x)\,dx.
\end{align*}
{}From H\"{o}lder's inequality, 
\begin{align*}
\lefteqn{
\sum_{\substack{Q: \\ \pi(Q)=F}}
\frac{\sg(Q)}{|Q|}
\int_{E(Q)}g(x)u(x)\,dx
}\\ &\le
\l(
\sum_{\substack{Q: \\ \pi(Q)=F}}
\l(\frac{\sg(Q)}{|Q|}\r)^q
\int_{E(Q)}u(x)\,dx
\r)^{1/q}
\l(
\sum_{\substack{Q: \\ \pi(Q)=F}}
\int_{E(Q)}g(x)^{q'}u(x)\,dx
\r)^{1/q'}.
\end{align*}
{}From the definition	of $M$, 
the facts that 
$E(Q)\subset Q$ and 
the sets $E(Q)$ are pairwise disjoint, 
$$
\le
\l(
\int_{F}M[\sg1_{F}](x)u(x)\,dx
\r)^{1/q}
\l(
\sum_{\substack{Q: \\ \pi(Q)=F}}
\int_{E(Q)}g(x)^{q'}u(x)\,dx
\r)^{1/q'}.
$$
{}From H\"{o}lder's inequality again, 
\eqref{3.3} can be majorized by 
\begin{align*}
2&\sum_{F\in\cF}
\fint_{F}f(y)\sg(y)^{-1}\,d\sg(y)
\l(\int_{F}M[\sg1_{F}](x)u(x)\,dx\r)^{1/q}
\l(
\sum_{\substack{Q: \\ \pi(Q)=F}}
\int_{E(Q)}g(x)^{q'}u(x)\,dx
\r)^{1/q'}
\\ &\le 2
\l\{
\sum_{F\in\cF}
\l(
\fint_{F}f(y)\sg(y)^{-1}\,d\sg(y)
\l(\int_{F}M[\sg1_{F}](x)^qu(x)\,dx\r)^{1/q}
\r)^q
\r\}^{1/q}
\\ &\quad\times
\l\{
\sum_{F\in\cF}
\sum_{\substack{Q: \\ \pi(Q)=F}}
\int_{E(Q)}g(x)^{q'}u(x)\,dx
\r\}^{1/q'}
\\ &=:
{\rm(i)}\times{\rm(ii)}.
\end{align*}
Since the sets $E(Q)$ are pairwise disjoint, 
$$
{\rm(ii)}
=
\l(\int_{Q_0}g(x)^{q'}u(x)\,dx\r)^{1/q'}
\le 1.
$$
Since $p\le q$ and 
$\|\cdot\|_{l^p}\ge\|\cdot\|_{l^q}$, 
\begin{equation}\label{3.6}
{\rm(i)}
\le
\l\{
\sum_{F\in\cF}
\l(
\fint_{F}f(y)\sg(y)^{-1}\,d\sg(y)
\l(\int_{F}M[\sg1_{F}](x)^qu(x)\,dx\r)^{1/q}
\r)^p
\r\}^{1/p}.
\end{equation}
Further, 
\begin{align*}
&\le
\l\{
\sup_{F\in\cF}
\frac1{\sg(F)^{1/p}}
\l(\int_{F}M[\sg1_{F}](x)^qu(x)\,dx\r)^{1/q}
\r\}
\\ &\quad\times
\l\{
\sum_{F\in\cF}
\l(\fint_{F}f(y)\sg(y)^{-1}\,d\sg(y)\r)^p
\sg(F)
\r\}^{1/p}
\\ &\le
c_3l(Q_0)^{\lm/q}
\l\{
\sum_{F\in\cF}
\l(\fint_{F}f(y)\sg(y)^{-1}\,d\sg(y)\r)^p
\sg(F)
\r\}^{1/p},
\end{align*}
where we have used \eqref{3.4}.

By using 
the definition of $M_{\sg}$, 
\eqref{3.5} and the facts that 
$E_{\cF}(F)\subset F$ and 
the sets $E_{\cF}(F)$ are pairwise disjoint, 
\begin{align*}
\lefteqn{
\l\{
\sum_{F\in\cF}
\l(\fint_{F}f(y)\sg(y)^{-1}\,d\sg(y)\r)^p
\sg(F)
\r\}^{1/p}
}\\ &\le C
\l\{
\sum_{F\in\cF}
\l(\fint_{F}f(y)\sg(y)^{-1}\,d\sg(y)\r)^p
\sg(E_{\cF}(F))
\r\}^{1/p}
\\ &\le C
\l(\int_{\R^n}M_{\sg}[f\sg^{-1}](x)^p\,d\sg(x)\r)^{1/p}.
\end{align*}
By use of Lemma \ref{lem3.2}, 
\begin{align*}
&\le C
\l(\int_{\R^n}[f(x)\sg(x)^{-1}]^p\sg(x)\,dx\r)^{1/p}
\\ &=C
\l(\int_{\R^n}f(x)^pb(x)v(x)\,dx\r)^{1/p}
\\ &\le C
\|f\|_{L^{p,\lm}(v)},
\end{align*}
where we have used \eqref{2.3}.

So altogether we obtain 
$$
\l(\int_{Q_0}M[f1_{Q_0}](x)^qu(x)\,dx\r)^{1/q}
\le C
c_3l(Q_0)^{\lm/q}
\|f\|_{L^{p,\lm}(v)}
$$
and
$$
{\rm\eqref{3.2}}
\le Cc_3
\|f\|_{L^{p,\lm}(v)}.
$$
These complete the proof of Theorem \ref{thm3.1} (II).

\subsection{Proof of Theorem \ref{thm3.1} (III)}\label{ssec3.3}
Assume that $1<p\le q<\8$ and 
the statements (b) and (d). 
Going through the same argument as before, 
retaining the same notation, 
we need only evaluate \eqref{3.2} 
especially \eqref{3.3}.
Letting $\sg\equiv 1$ in \eqref{3.6}, 
we see that 
$$
{\rm(i)}
\le
\l\{
\sum_{F\in\cF}
\l(\fint_{F}f(y)\,dyu(F)^{1/q}\r)^p
\r\}^{1/p}.
$$
Fix $b\in\cB_{\lm}$ so that 
\begin{equation}\label{3.7}
\frac1{l(Q_0)^{\lm/q}}
\sup_{\substack{Q\in\cQ \\ Q\subset Q_0}}
\frac{u(Q)^{1/q}}{|Q|^{1/p}}
\l(\fint_{Q}[b(x)v(x)]^{-ap'/p}\,dx\r)^{1/ap'}
\le 2c_4.
\end{equation}
Take $c<1$ is a number that satisfy $(cp)'=ap'$. 
H\"{o}lder's inequality gives 
\begin{align*}
\fint_{F}f(y)\,dy
&=
\fint_{F}f(y)[b(y)v(y)]^{1/p}[b(y)v(y)]^{-1/p}\,dy
\\ &\le
\l(\fint_{F}f(y)^{cp}[b(y)v(y)]^c\,dy\r)^{1/cp}
\l(\fint_{F}[b(y)v(y)]^{-ap'/p}\,dy\r)^{1/ap'},
\end{align*}
which implies 
\begin{align*}
{\rm(i)}
&\le
\l\{
\sup_{F\in\cF}
\frac{u(F)^{1/q}}{|F|^{1/p}}
\l(\fint_{F}[b(y)v(y)]^{-ap'/p}\,dy\r)^{1/ap'}
\r\}
\\ &\quad\times
\l\{
\sum_{F\in\cF}
\l(\fint_{F}f(y)^{cp}[b(y)v(y)]^c\,dy\r)^{1/c}
|F|
\r\}^{1/p}
\\ &\le 
c_4l(Q_0)^{\lm/q}
\l\{
\sum_{F\in\cF}
\l(\fint_{F}f(y)^{cp}[b(y)v(y)]^c\,dy\r)^{1/c}
|F|
\r\}^{1/p},
\end{align*}
where we have used \eqref{3.7}. 

The definition of $M$, 
the facts that 
$|F|\le 2|E_{\cF}(F)$,
$E_{\cF}(F)\subset F$ and 
the sets $E_{\cF}(F)$ are pairwise disjoint 
read 
\begin{align*}
\lefteqn{
\l\{
\sum_{F\in\cF}
\l(\fint_{F}f(y)^{cp}[b(y)v(y)]^c\,dy\r)^{1/c}
|F|
\r\}^{1/p}
}\\ &\le C
\l\{
\sum_{F\in\cF}
\l(\fint_{F}f(y)^{cp}[b(y)v(y)]^c\,dy\r)^{1/c}
|E_{\cF}(F)|
\r\}^{1/p}
\\ &\le C
\l(\int_{\R^n}M[f^{cp}(bv)^c](x)^{1/c}\,dx\r)^{1/p}
\\ &\le C
\l(\int_{\R^n}f(x)^pb(x)v(x)\,dx\r)^{1/p}
\\ &\le C
\|f\|_{L^{p,\lm}(v)},
\end{align*}
where we have used 
the $L^{1/c}$-boundedness of $M$ and 
\eqref{2.3}. 

So altogether we obtain 
$$
\l(\int_{Q_0}M[f1_{Q_0}](x)^qu(x)\,dx\r)^{1/q}
\le C
c_4l(Q_0)^{\lm/q}
\|f\|_{L^{p,\lm}(v)}
$$
and
$$
{\rm\eqref{3.2}}
\le Cc_4
\|f\|_{L^{p,\lm}(v)}.
$$
This completes the proof of Theorem \ref{thm3.1} (III).

\subsection{Proof of Theorem \ref{thm3.1} (IV)}\label{ssec3.4}
Assume that 
$0<q<p<\8$, $1<p<\8$ and 
the statements (b) and (e). 
In the same manner as above, 
retaining the same notation as before, 
we need only evaluate \eqref{3.2}. 
Fix $b\in\cB_{\lm}$ so that 
\begin{equation}\label{3.8}
\frac1{l(Q_0)^{\lm/q}}
\l(\int_{Q_0}
M[(bv)^{-p'/p}](x)^r
u(x)^{r/q}[b(x)v(x)]^{rp'/p^2}
\,dx\r)^{1/r}
\le 2c_5,
\quad\frac1q=\frac1r+\frac1p,
\end{equation}
and, let $\sg=(bv)^{-p'/p}$. 

We have for every $Q\in\cD(Q_0)$, 
\begin{align*}
\fint_{Q}f(y)\,dy
&=
\frac{\sg(Q)}{|Q|}
\fint_{Q}f(y)\sg(y)^{-1}\,d\sg(y)
\\ &\le
M[\sg](x)M_{\sg}[f\sg^{-1}](x),
\quad x\in Q,
\end{align*}
which implies
$$
Mf(x)
\le
M[\sg](x)M_{\sg}[f\sg^{-1}](x),
\quad x\in Q_0.
$$
Thus, 
\begin{align*}
\lefteqn{
\l(\int_{Q_0}M[f1_{Q_0}](x)^qu(x)\,dx\r)^{1/q}
\le
\l(\int_{Q_0}M[\sg](x)^qM_{\sg}[f\sg^{-1}](x)^qu(x)\,dx\r)^{1/q}
}\\ &=
\l(
\int_{Q_0}
M[\sg](x)^qu(x)\sg(x)^{-1}
\cdot
M_{\sg}[f\sg^{-1}](x)^q
\,d\sg(x)
\r)^{1/q}.
\end{align*}
{}From H"{o}lder's inequality with 
the exponent $(p-q)/p+q/p=1$ and 
the fact that $1/r=(p-q)/pq$, 
\begin{align*}
&\le
\l(
\int_{Q_0}
\l(M[\sg](x)^qu(x)\sg(x)^{-1}\r)^{p/(p-q)}
\,d\sg(x)\r)^{1/r}
\l(
\int_{Q_0}
M_{\sg}[f\sg^{-1}](x)^p\,d\sg(x)
\r)^{1/p}
\\ &=:
{\rm(iii)}\times{\rm(iv)}.
\end{align*}
We have by \eqref{3.8} 
$$
{\rm(iii)}
=
\l(
\int_{Q_0}
M[\sg](x)^ru(x)^{r/q}\sg(x)^{-r/p}
\,dx
\r)^{1/r}
\le
2c_5l(Q_0)^{\lm/q},
$$
and, we have by Lemma \ref{lem3.2} 
$$
{\rm(iv)}
\le C
\l(\int_{\R^n}f(x)^pb(x)v(x)\,dx\r)^{1/p}
\le C
\|f\|_{L^{p,\lm}(v)}.
$$
These imply 
$$
\l(\int_{Q_0}M[f1_{Q_0}](x)^qu(x)\,dx\r)^{1/q}
\le C
c_5l(Q_0)^{\lm/q}
\|f\|_{L^{p,\lm}(v)}
$$
and
$$
{\rm\eqref{3.2}}
\le Cc_5
\|f\|_{L^{p,\lm}(v)}.
$$
This completes the proof of Theorem \ref{thm3.1} (IV).

\section{One-weight norm inequalities}\label{sec4}
We restate Theorem \ref{thm3.1} in terms of the one-weight setting. 

\begin{proposition}\label{prp4.1}
Let $1<p<\8$, $0<\lm<n$ 
and $w$ be a weight. 
Consider the following four statements: 

\begin{enumerate}
\item[{\rm(a)}] 
There exists a constant $c_1>0$ 
such that 
$$
\|Mf\|_{L^{p,\lm}(w)}
\le c_1
\|f\|_{L^{p,\lm}(w)}
$$
holds for every function 
$f\in L^{p,\lm}(\R^n,w)$;
\item[{\rm(b)}] 
There exists a constant $c_2>0$ 
such that 
$$
\frac1{|Q|}
\|w^{1/p}1_{Q}\|_{L^{p,\lm}}
\|w^{-1/p}1_{Q}\|_{H^{p',\lm}}
\le c_2
$$
holds for every cube $Q\in\cQ$;
\item[{\rm(c)}] 
There exists a constant $c_3>0$ 
such that 
$$
\inf_{b\in\cB_{\lm}}
\l(
\frac1{l(Q_0)^{\lm}}
\sup_{\substack{Q\in\cQ \\ Q\subset Q_0}}
\frac1{[(bw)^{-p'/p}](Q)}
\int_{Q}M[(bw)^{-p'/p}1_{Q}](x)^pw(x)\,dx
\r)
\le c_3^p
$$
holds for every cube $Q_0\in\cQ$;
\item[{\rm(d)}] 
There exists a constant $c_4>0$ 
such that, for some $a>1$, 
$$
\inf_{b\in\cB_{\lm}}
\l(
\frac1{l(Q_0)^{\lm}}
\sup_{\substack{Q\in\cQ \\ Q\subset Q_0}}
\frac{w(Q)}{|Q|}
\l(\fint_{Q}[b(x)w(x)]^{-ap'/p}\,dx\r)^{p/ap'}
\r)
\le c_4^p
$$
holds for every cube $Q_0\in\cQ$.
\end{enumerate}

\noindent
Then,

\begin{enumerate}
\item[{\rm(I)}] 
{\rm(a)} implies {\rm(b)} 
with $c_2\le Cc_1$;
\item[{\rm(II)}] 
{\rm(b)} and {\rm(c)} imply {\rm(a)} 
with $c_1\le C(c_2+c_3)$;
\item[{\rm(III)}] 
{\rm(b)} and {\rm(d)} imply {\rm(a)} 
with $c_1\le C(c_2+c_4)$.
\end{enumerate}
\end{proposition}

We have from this proposition the following.

\begin{proposition}\label{prp4.2}
Let $1<p<\8$, $0<\lm<n$ and 
$w=|\cdot|^{\al}$ be a power weight. 
Then, the weighted inequality 
$$
\|Mf\|_{L^{p,\lm}(w)}
\le C
\|f\|_{L^{p,\lm}(w)}
$$
holds if and only if 
$\lm-n\le\al<\lm+(p-1)n$.
\end{proposition}

\begin{proof}
Assume that $\lm-n\le\al<\lm+(p-1)n$.
We first evaluate
\begin{equation}\label{4.1}
\frac1{|Q_0|}
\|w^{1/p}1_{Q_0}\|_{L^{p,\lm}}
\|w^{-1/p}1_{Q_0}\|_{H^{p',\lm}}
\end{equation}
for 
$$
Q_0=\l(c,c+\frac{d}{\sqrt{n}}\r)\times\l(0,\frac{d}{\sqrt{n}}\r)^{n-1}\subset\R^n,
\quad c,d>0.
$$

Suppose that $d\le c$. 
Let $0<\lm<\lm_0<n$ and set 
$$
b_1(x)
=
C\frac{M[1_{Q_0}](x)^{\lm_0/n}}{l(Q_0)^{\lm}}.
$$
Then, 
we see that $b_1$ belongs to $\cB_{\lm}$ 
(see Section \ref{sec2}). This implies 
\begin{align*}
{\rm\eqref{4.1}}
&\le
\frac1{|Q_0|}
\l(\int_{Q_0}l(Q_0)^{-\lm}|x|^{\al}\,dx\r)^{1/p}
\l(\int_{Q_0}[|x|^{\al}b_1(x)]^{-p'/p}\,dx\r)^{1/p'}
\\ &=
\frac1{|Q_0|}
\l(\int_{Q_0}l(Q_0)^{-\lm}|x|^{\al}\,dx\r)^{1/p}
\l(\int_{Q_0}[|x|^{\al}l(Q_0)^{-\lm}]^{-p'/p}\,dx\r)^{1/p'}
\\ &\le
\l(\frac{\sup_{x\in Q_0}|x|^{\al}}{\inf_{x\in Q_0}|x|^{\al}}\r)^{1/p}
\le 
\l(\frac{c+d}{c}\r)^{|\al|/p}
\le C.
\end{align*}

Suppose that $d>c$. Let 
$B_0=\{x\in\R^n:\,|x|<2d\}$, 
$\al<\lm_1+(p-1)n<\lm+(p-1)n$. 
and set 
$$
b_2(x)
=
\frac{\lm/\lm_1-1}{\lm/\lm_1}
(2d)^{\lm_1-\lm}
|x|^{-\lm_1}1_{B_0}(x).
$$
Then, we see that 
$|x|^{-\lm_1}\in A_1$ 
and 
$\ds\int_{\R^n}b_2\,dH^{\lm}=1$.
Indeed, 
\begin{align*}
\lefteqn{
\int_{\R^n}|\cdot|^{-\lm_1}1_{B_0}\,dH^{\lm}
}\\ &=
(2d)^{\lm-\lm_1}
+
\int_{(2d)^{-\lm_1}}^{\8}
t^{-\lm/\lm_1}\,dt
\\ &=
(2d)^{\lm-\lm_1}
+
\frac1{\lm/\lm_1-1}
(2d)^{\lm-\lm_1}
=
\frac{\lm/\lm_1}{\lm/\lm_1-1}
(2d)^{\lm-\lm_1}.
\end{align*}
Thus, we obtain 
\begin{align*}
{\rm\eqref{4.1}}
&\le
\frac1{|Q_0|}
\l(\int_{Q_0}l(Q_0)^{-\lm}|x|^{\al}\,dx\r)^{1/p}
\l(\int_{Q_0}[|x|^{\al}b_2(x)]^{-p'/p}\,dx\r)^{1/p'}
\\ &\le C
\frac1{|B_0|}
\l(\int_{B_0}(2d)^{-\lm}|x|^{\al}\,dx\r)^{1/p}
\l(\int_{B_0}[|x|^{\al}b_2(x)]^{-p'/p}\,dx\r)^{1/p'}
\\ &=C
\frac1{|B_0|}
(2d)^{-\lm_1/p}
\l(\int_{B_0}|x|^{\al}\,dx\r)^{1/p}
\l(\int_{B_0}|x|^{\frac{\lm_1-\al}{p-1}}\,dx\r)^{1/p'}
\\ &\le C,
\end{align*}
where we have used 
$\al<\lm_1+(p-1)n$. 

Next, we evaluate 
\begin{equation}\label{4.2}
\inf_{b\in\cB_{\lm}}
\l(
\frac1{l(Q_0)^{\lm}}
\sup_{\substack{Q\in\cQ \\ Q\subset Q_0}}
\frac{w(Q)}{|Q|}
\l(\fint_{Q}[b(x)w(x)]^{-ap'/p}\,dx\r)^{p/ap'}
\r).
\end{equation}
When $d\le c$, 
the same estimates of \eqref{4.1} are available to those of \eqref{4.2}.
When $d>c$ and 
$\ds
\frac{d_0}{c_0}
:=
\frac{\sup_{x\in Q}|x|}{\inf_{x\in Q}|x|}
\le 2$,
\begin{align*}
\lefteqn{
\frac1{l(Q_0)^{\lm}}
\frac{w(Q)}{|Q|}
\l(\fint_{Q}[b_2(x)w(x)]^{-ap'/p}\,dx\r)^{p/ap'}
}\\ &\le
\frac{(d_0)^{\lm_1}}{(2d)^{\lm_1}}
\l(\frac{d_0}{c_0}\r)^{|\al|}
\le 2.
\end{align*}
When $d>c$ and 
$\ds\frac{\sup_{x\in Q}|x|}{\inf_{x\in Q}|x|}>2$,
the same estimates of \eqref{4.1} are available too.

If $\al<\lm-n$, then 
$$
\|w^{1/p}1_{(0,1)^n}\|_{L^{p,\lm}}=\8,
$$
and, if $\al\ge\lm+(p-1)n$, then 
$$
\|w^{-1/p}1_{(0,1)^n}\|_{H^{p',\lm}}=8.
$$
These yield the proposition.
\end{proof}

\section{Appendix}\label{sec5}
As an appendix, we shall show 
the following two-weight norm inequality in the upper triangle case 
$0<q<p<\8$, $1<p<\8$. 

\begin{proposition}\label{prp5.1}
Let $0<q<p<\8$, $1<p<\8$ 
and $u,v$ be weights. 
Suppose that $v\in A_1$. Then, 
the weighted inequality 
\begin{equation}\label{5.1}
\|Mf\|_{L^q(u)}
\le C
\|f\|_{L^p(v^{1-p})}
\end{equation}
holds if and only if
\begin{equation}\label{5.2}
\|u^{1/q}v^{1/p'}\|_{L^r}<\8,
\quad\frac1q=\frac1r+\frac1p.
\end{equation}
\end{proposition}

\begin{proof}
In the same manner as in the proof of Theorem \ref{thm3.1}, 
we may assume that $f$ is non-negative and 
$M$ is the dyadic maximal operator.

Suppose that \eqref{5.2} holds.
We have for every $Q\in\cD$, 
\begin{align*}
\fint_{Q}f(y)\,dy
&=
\frac{v(Q)}{|Q|}
\fint_{Q}f(y)v(y)^{-1}\,dv(y)
\\ &\le
Mv(x)M_v[fv^{-1}](x)
\le C
v(x)M_v[fv^{-1}](x),
\quad x\in Q,
\end{align*}
where we have used $v\in A_1$. 
This implies
$$
Mf(x)\le Cv(x)M_v[fv^{-1}](x),
\quad x\in\R^n.
$$
Thus, 
\begin{align*}
\lefteqn{
\l(\int_{\R^n}Mf(x)^qu(x)\,dx\r)^{1/q}
\le C
\l(\int_{\R^n}v(x)^qM_v[fv^{-1}](x)^qu(x)\,dx\r)^{1/q}
}\\ &=C
\l(
\int_{\R^n}
v(x)^{q-1}u(x)
\cdot
M_v[fv^{-1}](x)^q
\,dv(x)
\r)^{1/q}.
\end{align*}
{}From H"{o}lder's inequality with 
the exponent $(p-q)/p+q/p=1$ and 
the fact that $1/r=(p-q)/pq$, 
\begin{align*}
&\le C
\l(
\int_{\R^n}
\l(v(x)^{q-1}u(x)\r)^{p/(p-q)}
\,dv(x)\r)^{1/r}
\l(
\int_{\R^n}
M_v[fv^{-1}](x)^p\,dv(x)
\r)^{1/p}
\\ &\le C
\l(
\int_{\R^n}
\l(u(x)^{1/q}v(x)^{1-1/q+1/r}\r)^r
\,dx\r)^{1/r}
\\ &\quad\times
\l(\int_{\R^n}f(x)^pv(x)^{1-p}\,dx\r)^{1/p}
\le C
\|u^{1/q}v^{1/p'}\|_{L^r}
\|f\|_{L^p(v^{1-p})},
\end{align*}
where we have used Lemma \ref{lem3.2}.

Suppose that \eqref{5.1} holds.
Notice that $q/p+q/r=1$.
Keeping this in mind, we evaluate 
\begin{equation}\label{5.3}
\int_{\R^n}g(x)v(x)^{q/p'}u(x)\,dx
\end{equation}
with a non-negative function $g$ 
which satisfies $\|g\|_{L^{p/q}}\le 1$.

It follows from \eqref{5.1} that 
\begin{align*}
{\rm\eqref{5.3}}
&=
\int_{\R^n}[g(x)^{1/q}v(x)^{1/p'}]^qu(x)\,dx
\\ &\le
\int_{\R^n}M[g^{1/q}v^{1/p'}](x)^qu(x)\,dx
\\ &le C
\l(
\int_{\R^n}
g(x)^{p/q}v(x)^{p/p'}v(x)^{1-p}
\,dx\r)^{q/p}
\\ &=C
\l(\int_{\R^n}g(x)^{p/q}\,dx\r)^{q/p}
\le C.
\end{align*}
This yields \eqref{5.2}.
\end{proof}

\end{document}